\documentclass[11pt]{article}
\usepackage[reqno]{amsmath}
\usepackage{amssymb,amscd,latexsym,paralist}
\usepackage{bm}
\usepackage[all]{xy}
\newcommand{\A}{{\mathbb{A}}}
\newcommand{\C}{{\mathbb{C}}}

\newcommand{\G}{\mathbb{G}}
\newcommand{\Pa}{{\mathbb{P}}}
\newcommand{\Q}{{\mathbb{Q}}}
\newcommand{\oQ}{\overline{\Q}}

\newcommand{\Z}{{\mathbb{Z}}}
\newcommand{\hZ}{\hat{\Z}}

\newcommand{\bo}{\mathbf{1}}

\newcommand{\et}{\mathrm{\acute{e}t}}

\newcommand{\Ind}{\mathrm{Ind}}

\newcommand{\ok}{\overline{k}}
\newcommand{\oP}{\overline{P}}
\newcommand{\red}{\mathrm{red}}

\newcommand{\spec}{\mathrm{spec}\,}
\newcommand{\Aut}{\mathrm{Aut}}
\newcommand{\Coker}{\mathrm{Coker}\,}

\newcommand{\Gal}{\mathrm{Gal}}
\newcommand{\GL}{\mathrm{GL}}

\newcommand{\Ker}{\mathrm{Ker}\,}

\newcommand{\tzeta}{\tilde{\zeta}}

\newcommand{\RRe}{\mathrm{Re}\,}

\newcommand{\Tr}{\mathrm{Tr}}

\newcommand{\Ah}{{\mathcal A}}

\newcommand{\Gh}{{\mathcal G}}
\newcommand{\Hh}{{\mathcal H}}

\newcommand{\Uh}{\mathcal{U}}

\newcommand{\eq}{{\mathfrak{q}}}

\newcommand{\eo}{\mathfrak{o}}
\newcommand{\oeo}{\overline{\eo}}

\newcommand{\ep}{\mathfrak{p}}
\newcommand{\eP}{\mathfrak{P}}

\newcommand{\eX}{{\mathfrak X}}

\newcommand{\oK}{\overline{K}}
\newcommand{\silo}{\xrightarrow{\sim}}
\newcommand{\tei}{\, | \,}
\newenvironment{theorem}{\noindent {\bf Theorem}\it}{}

\newenvironment{examp}{\noindent {\bf Example}}{}

\newenvironment{proof}{\noindent {\bf Proof}}{\mbox{}\hfill$\Box$}
\begin{document}
\title{Horizontal factorizations of certain Hasse--Weil zeta functions --- a remark on a paper by Taniyama\\
{\large \it Dedicated to Francesco Baldassarri}}
\author{Christopher Deninger \and Dimitri Wegner}
\date{\ }
\maketitle
\section{Introduction}
The paper \cite{T} by Taniyama contains several results in the arithmetic theory of $CM$ abelian varieties which have become well known. Moreover the notion of a compatible system $\rho = (\rho_l)$ of $l$-adic representations is introduced there for the first time. Under suitable conditions on $\rho$, Taniyama proves an interesting formula for the alternating product of the $L$-functions of the exterior powers of $\rho$: it is given by an infinite product of Artin $L$-functions each of which modified by a change of finitely many Euler factors. As an application, Taniyama obtains such a formula for the Hasse--Weil zeta function of an abelian scheme over a localized number ring. \\
These infinite product formulas of Taniyama seem to be little known but they motivated the works \cite{JY} and \cite{JR}. In \cite{JY} Joshi and Yogananda show that the Hasse--Weil zeta function $\zeta_{\G_m} (s) = \zeta (s-1) / \zeta (s)$ of $\G_m = \spec \Z [T,T^{-1}]$ is an infinite product of Dirichlet $L$-series. This is a special case of Taniyama's formulas but they give an elementary direct argument. In \cite{JR} Joshi and Raghunathan express quotients of very general  Dirichlet series with Euler products as infinite products of ``twisted Dirichlet series''. The method is purely local. It applies in particular to the Hasse--Weil zeta function of $\eX \times \G_m$ where $\eX$ is a scheme of finite type over $\Z$. In this case the formula has a geometric proof which was suggested by Serre, c.f. \cite{JR} \S\,3. Joshi and Raghunathan also obtain certain infinite product formulas for quotients of automorphic $L$-functions. Recently these have been related to canonical bases by Kim and Lee, \cite{KL} Remark 2.21.

In the present note we interpret Taniyama's identity for the Hasse--Weil zeta function of an abelian scheme $\Ah$ as a ``horizontal factorization'': We write $\Ah$ as an essentially disjoint union of horizontal prime divisors corresponding to certain closed points $P$ on the generic fibre $A = \Ah \otimes \Q$. The residue field of each point $P$ is a number field $K = K (P)$ and $\zeta_{\Ah} (s)$ is essentially the product of the corresponding Dedekind zeta-functions $\zeta_K (s)$. To get an exact formula one needs to change each $\zeta_K (s)$ at finitely many Euler factors. Our method is similar to the one suggested by Serre and used in \cite{JR} \S\,3 and surely our interpretation of Taniyama's formula is known to Serre. We found it before learning about \cite{JR}.\\ 
Section 2 contains a review of Taniyama's formula in the context of $l$-adic representations and a discussion of the examples coming from the $l$-adic cohomologies of $\G_m$ and abelian varieties.\\
In section 3 we prove the ``horizontal factorization'' of the Hasse--Weil zeta functions of a class of group schemes $\Gh$ containing all semi-abelian schemes. The method is geometric and elementary. For $\G_m$ and abelian schemes the factorization formulas are the same as the ones in section 2 which follow from Taniyama's theorem. This clarifies the meaning of Taniyama's formulas in these cases.\\
In section 4 we give some generalizations. In particular there are ``horizontal factorizations'' for the zeta functions of arbitrary open subschemes in $\Pa^N$ over $\spec \Z$.

The first author would like to thank Dinakar Ramakrishnan for having shown him the infinite product formula in Taniyama's work many years ago.  
\section{Taniyama's product formula}
We begin by recalling a result of Taniyama using the terminology of \cite{S} Ch. I which developed from \cite{T}. Let $K / \Q$ be an algebraic number field with ring of integers $\eo_K$ and absolute Galois group $G_K = \Gal (\oK / K)$. We fix algebraic closures and embeddings $\oQ \subset \C$ and $\oQ \subset \oQ_l$ for each prime number $l$. Consider a strictly compatible system of integral $l$-adic representations $\rho = (\rho_l)$ of $G_K$ on free $\Z_l$-modules $E_l$ of rank $d \ge 1$. Let $S$ be the exceptional set of $\rho$. It is the smallest set of places of $K$ such that for $\ep \notin S$ the representation $\rho_l$ is unramified at $\ep$ for each $\ep \nmid l$. For a finite place $\ep$ of $K$ and a place $\eP$ of $\oK$ extending $\ep$ let $I_{\eP}\subset G_{\eP} \subset G_K$ be the inertia and decomposition groups at $\eP$ and denote by $F_{\eP} \in G_{\eP} / I_{\eP}$ the (arithmetic) Frobenius element. If $\ep \notin S , \ep \nmid l$ the automorphism $F_{\eP , \rho_l} = \rho_l (F_{\eP})$ of $E_l$ is well defined. Let $F_{\ep, \rho_l}$ be its conjugacy class in $\Aut \, E_l$. By assumption, 
\[
 P_{\ep, \rho} (T) = \det (1 - F_{\ep , \rho_l} T \tei E_l) \quad \text{and} \quad Q_{\ep,\rho} (T) = \det (F_{\ep , \rho_l} - T \mid E_l)
\]
are polynomials of degree $d$ in $\Z [T]$ which do not depend on $l$. It follows that $\det F_{\ep , \rho_l}$ is an integer independent of $l$. In this situation, Taniyama makes further assumptions:

(I) There is a non-negative integer $w$ independent of $\ep \notin S$ such that all zeroes of $Q_{\ep , \rho} (T)$ in $\C$ have absolute value $N\ep^{w/2}$.

(II) $\det F_{\ep , \rho_l}$ is positive i.e. $P_{\ep , \rho} (T) = (-1)^d a_d T^d + \ldots + 1$ with $a_d > 0$ for any $\ep \notin S$.

(III) For $\ep \notin S$ with $\ep \tei l$ and any place $\eP \tei \ep$ view $E_l$ as a $G_{\eP}$-module. Then there are $l$-adic representations 
\[
 \rho^0_l : G_{\eP} \to \Aut (E^0_l) \quad \text{and} \quad \rho^{\et}_l : G_{\eP} \to \Aut (E^{\et}_l)
\] 
on free $\Z_l$-modules $E^0_l$ and $E^{\et}_l$ with the following properties:

(1) There is an exact sequence of $G_{\eP}$-modules
\[
 0 \longrightarrow E^0_l \longrightarrow E_l \longrightarrow E^{\et}_l \longrightarrow 0 \; .
\]
(2) The representation $\rho^{\et}_l$ of $G_{\eP}$ is unramified.

(3) The eigenvalues of $\rho^{\et}_l (F_{\eP})$ in $\oQ_l$ are exactly those zeroes of $Q_{\ep , \rho} (T)$ under the fixed embedding $\oQ \subset \oQ_l$ which are $l$-adic units.

The representation spaces $E^0_l$ and $E^{\et}_l$ depend on $\ep$. In particular their ranks may vary with $\ep$. Condition (III) needs to be checked for one place $\eP \tei \ep$ only.

Consider the $L$-function of $\rho$
\[
 L (\rho,s) = \prod_{\ep \notin S} P_{\ep,\rho} (N\ep^{-s})^{-1} \; .
\]
By assumption it converges locally uniformly in $\RRe s > 1 + \frac{w}{2}$. More generally, the $L$-function of $\Lambda^i \rho$ converges in $\RRe s > 1 + \frac{iw}{2}$. In the theorem recalled below, Taniyama expresses the alternating product of the $L$-series $L (\Lambda^i \rho,s)$ as an infinite product of modified Artin $L$-series. They are defined as follows: Consider the free $\hZ$-module $E = \prod_l E_l$ of rank $d$. It is a module under the Galois group $G_K$. For a place $\ep \notin S$ of $K$ and a place $\eP \tei \ep$ we also need the $\hZ$-module
\[
 E_{\ep} = \prod_{\ep \nmid l} E_l \times \prod_{\ep \tei l} E^{\et}_l
\]
with its $G_{\eP}$-action via $\rho_l$ and $\rho^{\et}_l$. For any integer $n \ge 1$ the finite $\Z / n = \hZ / n\hZ$-modules $E (n) = E / nE$ and $E_{\ep} (n) = E_{\ep} / n E_{\ep}$ carry $G_K$- resp. $G_{\eP}$-operations and for $\ep \nmid n$ we have $E_{\ep} (n) = E (n)$. By $E_{\ep} (n)^*$ and $E (n)^*$ we denote the subsets of elements of order $n$ in $E_{\ep} (n)$ resp. $E(n)$. They are $G_{\eP}$- resp. $G_K$-invariant. For $\ep \notin S$ and $\nu \ge 1$, Taniyama defines
\[
 \psi_n (\ep^{\nu}) := \Tr (F^{\nu}_{\eP} \tei \C [E_{\ep} (n)^*]) \; .
\]
This is independent of the choice of $\eP \tei \ep$ and it follows from the Chinese remainder theorem that for coprime $n,m$ we have
\begin{equation} \label{eq:1}
 \psi_{nm} (\ep^{\nu}) = \psi_n (\ep^{\nu}) \psi_m (\ep^{\nu}) \quad \text{if} \; \ep \notin S \; .
\end{equation}
We may also consider the Artin representation of $G_K$ on $\C [E (n)^*]$. For its character we have in particular:
\[
 \varphi_n (\ep^{\nu}) = \Tr (F^{\nu}_{\eP} \tei \C [E (n)^*] )\; .
\]
This shows: 
\begin{equation} \label{eq:2}
 \psi_n (\ep^{\nu}) = \varphi_n (\ep^{\nu}) \quad \text{for all $\ep \notin S$ with $\ep \nmid n$.}
\end{equation}
The $L$-function of $\psi_n$ is defined by the formula
\begin{align*}
 L (\psi_n , s) & = \exp \sum_{\ep \notin S} \sum^{\infty}_{\nu=1} \frac{\psi_n (\ep^{\nu})}{\nu} N \ep^{-\nu s} \\
& = \prod_{\ep \notin S} \det (1 - F_{\eP} N \ep^{-s} \tei \C [E_{\ep} (n)^*])^{-1} \; .
\end{align*}
The Artin $L$-function of $\varphi_n$ is given by the formula
\begin{align*}
 L (\varphi_n ,s) & = \exp \sum_{\ep \notin S} \sum^{\infty}_{\nu=1} \frac{\varphi_n (\ep^{\nu})}{\nu} N \ep^{-\nu s} \\
& = \prod_{\ep \notin S} \det (1 - F_{\eP} N \ep^{-s} \tei \C [E (n)^*])^{-1} \; .
\end{align*}
The Euler-factors of $L (\psi_n , s)$ and $L (\varphi_n, s)$ are the same for all $\ep \notin S , \ep \nmid n$. Hence $L (\psi_n , s)$ is an Artin $L$-series up to finitely many Euler factors. Taniyama's result is the following, \cite{T} Theorem 3. For the reader's convenience we sketch a proof.

\begin{theorem} {\bf (Taniyama)} \label{t1}
 Under the previous assumptions the following equality holds in $\RRe s > 1 + \frac{dw}{2}$
\begin{equation} \label{eq:3}
 \prod^d_{i=0} L (\Lambda^i \rho,s)^{(-1)^{d-i}} = \prod^{\infty}_{n=1} L (\psi_n , s) \; .
\end{equation}
\end{theorem}

\begin{proof}
 Taking logarithmic derivatives on both sides we are reduced to showing that for any $\ep \notin S$ we have an equality where the sum on the right is finite
\[
 \prod^d_{i=1} (\lambda^{\nu}_i - 1) = \sum^{\infty}_{n=1} \psi_n (\ep^{\nu}) \quad \text{for} \; \nu \ge 1 \; .
\]
Here $\lambda_1 , \ldots , \lambda_d$ are the zeroes of $Q_{\ep,\rho} (T)$. Since $\psi_n (\ep^{\nu})$ is multiplicative in $n$ we need to see that
\begin{equation}
 \label{eq:4}
\prod^d_{i=1} (\lambda^{\nu}_i -1) = \prod_l \sum^{\infty}_{k=0} \psi_{l^k} (\ep^{\nu}) \; .
\end{equation}
Here the sums on the right should be finite and equal to $1$ for almost all $l$.
The left hand side is in $\Z$ and an elementary argument using assumptions (I) and (II) shows that it is positive. As for the right hand side, if $\ep \nmid l$ we have a commutative diagram
\[
 \xymatrix{
0 \ar[r] & E_l \ar[r]^{l^k} \ar[d]^{F^{\nu}_{\eP}-1} & E_l \ar[r] \ar[d]^{F^{\nu}_{\eP} -1} & E_{\ep} (l^k) \ar[r] \ar[d]^{F^{\nu}_{\eP} -1} & 0 \\
0 \ar[r] & E_l \ar[r]^{l^k} & E_l \ar[r] & E_{\ep} (l^k) \ar[r] & 0 \; .
}
\]
By assumption (I) the middle and left vertical maps are injective with finite cokernels of $l$-power order. The snake lemma gives an isomorphism
\[
\Ker (F^{\nu}_{\eP}-1 \mid E_{\ep} (l^k)) \xrightarrow{\sim} \Coker (F^{\nu}_{\eP}-1 \mid E_l)_{l^k}  \; .
\]
By definition, $\psi_{l^k} (\ep^{\nu})$ is the number of elements of order $l^k$ in the group on the left. Hence $\psi_{l^k} (\ep^{\nu}) = 0$ for large $k$ and we have:
\begin{align*}
 \sum^{\infty}_{k=0} \psi_{l^k} (\ep^{\nu}) & = |\Coker (F^{\nu}_{\eP} -1 \mid E_l)|\\
& = |\det (F^{\nu}_{\eP} -1 \mid E_l)|^{-1}_l\\
& = |\prod^d_{i=1} (\lambda^{\nu}_i -1) |^{-1}_l \; .
\end{align*}
For prime numbers $l$ with $\ep \mid l$ we argue similarly replacing $E_l$ by $E^{\et}_l$. This gives
\[
 \sum^{\infty}_{k=0} \psi_{l^k} (\ep^{\nu}) = |\det (F^{\nu}_{\eP} -1 \mid E^{\et}_l) |^{-1}_l \; .
\]
 By assumption (III\,3) the eigenvalues $\lambda_i$ which do not occur among the eigenvalues of $F_{\eP}$ on $E^{\et}_l$ are those with positive $l$-valuation. For them $\lambda^{\nu}_i -1$ is an $l$-adic unit and therefore we have
\[
 \sum^{\infty}_{k=0} \psi_{l^k} (\ep^{\nu}) = |\prod^d_{i=1} (\lambda^{\nu}_i -1) |^{-1}_l
\]
in the case $\ep \mid l$ as well. Again the sum is finite. Equation \eqref{eq:4} now follows from the product formula.
\end{proof}

In order to understand Taniyama's formula better, note that the Artin $L$-series $L (\varphi_n, s)$ is a product of finitely many Dedekind zeta-functions: Decompose the $G_K$-set $E (n)^*$ into orbits
\[
 E(n)^* = \coprod_P G_K P
\]
where $P$ runs over a set of representatives in $E (n)^*$ of the $G_K$-orbits. Let $K_P$ be the fixed field of the stabilizer of $P$ in $G_K$ i.e. $G_{K_P} = (G_K)_P$. Then we have
\begin{align*}
 \C [E (n)^*] & = \bigoplus_P \C [G_K P] = \bigoplus_P \C [G_K / G_{K_P}] \\
& = \bigoplus_P \Ind^{G_K}_{G_{K_P}} (\bo) \; .
\end{align*}
The formalism of Artin $L$-series \cite{N} Ch. 7, \S\,10 now implies the formula
\begin{equation}
 \label{eq:5}
L (\varphi_n, s) = \prod_P \zeta_{K_P,S} (s) \; .
\end{equation}
Here $\zeta_{K_P,S} (s)$ is the Dedekind zeta function of $K_P$ without the Euler factors for primes above $S$.

We now discuss the cyclotomic character and following Taniyama also the representation of $G_K$ on the Tate-module of an abelian variety.

The group $G_K$ acts on $T_l \G_m = \varprojlim_{k} \mu_{l^k} (\oK)$. The corresponding representations $\rho_l : G_K \to \GL_1 (T_l \G_m) = \Z^{\times}_l$ form a strictly compatible system $\rho = (\rho_l)$ with empty exceptional set $S$. We have $P_{\ep, \rho} (T) = 1 - N\ep T$. Hence conditions (I) and (II) are satisfied with $d = 1 , w = 2$. Condition (III) holds with $E^{\et}_l = 0$ and $E^0_l = T_l \G_m$. The $L$-function of $\rho$ is $L (\rho , s) = \zeta_K (s-1)$. We have $E_{\ep} = \prod_{\ep \nmid l} T_l \G_m$ and hence $E_{\ep} (n) = \mu_{n_{\ep}} (\oK)$ where $n_{\ep}$ is the biggest prime to $\ep$ factor of $n$. It follows that $E_{\ep} (n)^* = \emptyset$ hence $\psi_n (\ep^{\nu}) = 0$ if $\ep \mid n$ and that $E_{\ep} (n)^*$ is the set of primitive $n$-th roots of unity in $\oK$ if $\ep \nmid n$. 

For $\ep \nmid n$ choosing a primitive $n$-th root of unity $\zeta_n$, the map $\nu \mapsto \zeta^{\nu}_n$ gives a bijection $(\Z / n)^{\times} \silo E_{\ep} (n)^*$. Multiplication by $N \ep$ on the left corresponds to the operation by $F_{\eP}$ on the right. It follows that
\begin{align*}
 \psi_n (\ep^{\nu}) & = \Tr (F^{\nu}_{\eP} \mid  \C [E_{\ep} (n)^*]) \\
& = \Tr (N \ep^{\nu} \mid \C [(\Z / n)^{\times}])\\
& = \sum_{\chi \in (\widehat{\Z / n)^{\times}}} \chi (N \ep^{\nu}) \; .
\end{align*}
Here we have used the decomposition of the regular representation of $(\Z / n)^{\times}$ into the direct sum of the characters $\chi$ of $(\Z / n)^{\times}$. Setting
\begin{equation}
 \label{eq:6}
L_K (\chi,s) = \prod_{\ep \nmid n} (1 - \chi (N \ep) N \ep^{-s})^{-1}
\end{equation}
we therefore find the formula
\begin{equation}
 \label{eq:7}
L (\psi_n, s) = \prod_{\chi \in \widehat{(\Z / n)^{\times}}} L_K (\chi,s) \; .
\end{equation}
Taniyama's theorem therefore implies the following relation:
\begin{equation}
 \label{eq:8}
\zeta_K (s-1) / \zeta_K (s) = \prod^{\infty}_{n=1} \prod_{\chi \in \widehat{(\Z / n)^{\times}}} L_K (\chi,s) \quad \text{in} \; \RRe s > 2 \; .
\end{equation}
For $K = \Q$, a direct proof of this formula is given in \cite{JY}. We can make formula \eqref{eq:5} more explicit in the example. The set $E(n)$ is the set of $n$-th roots of unity. The $G_K$-orbits on $E (n)$ are in bijection with the set $|\mu_{n,K}|$ of closed points $P$ of the finite group scheme $\mu_{n,K} = \spec K [T] / (T^n-1)$ and $K_P$ is the residue field of $P$, a finite extension of $K$. Thus the $G_K$-orbits on $E (n)^*$ are in bijection with
\[
 |\mu^*_{n,K}| := |\mu_{n,K}| \setminus \bigcup_{d \mid n \atop d \neq n} |\mu_{d,K}| \; .
\]
According to \eqref{eq:5}, we have
\begin{equation}
 \label{eq:9}
L (\varphi_n , s) = \prod_{P \in |\mu^*_{n,K}| }\zeta_{K_P} (s) \; .
\end{equation}
The $L$-functions $L (\varphi_n,s)$ and $L (\psi_n,s)$ have the same Euler factors for $\ep \nmid n$. For $\ep \mid n$ the Euler factor of $L (\psi_n,s)$ equals $1$. Accordingly we modify $\zeta_{K_P} (s)$ for $P \in |\mu^*_{n,K}|$ by removing the Euler factors for primes dividing $n$. Let $\tzeta_{K_P} (s)$ be the resulting zeta function
\[
 \tzeta_{K_P} (s) = \prod_{\eq \nmid n} (1 - N\eq^{-s})^{-1} 
\]
where $\eq$ runs over the prime ideals of the ring of integers in $K_P$. Then \eqref{eq:9} implies the formula
\begin{equation}
 \label{eq:10}
L (\psi_n,s) = \prod_{P \in |\mu^*_{n,K}|} \tzeta_{K_P} (s) \; .
\end{equation}
Using \eqref{eq:8} and the relation $\zeta_{\G_{m,\eo_K}} (s) = \zeta_K (s-1) / \zeta_K (s)$ we get
\begin{equation}
 \label{eq:11}
\zeta_{\G_{m,\eo_K}} (s) = \prod_P \tzeta_{K_P} (s) \; .
\end{equation}
Here $P$ runs over the closed points of $\G_{m,K}$ which are torsion i.e. correspond to Galois orbits of torsion points in $\G_{m,K} (\oK)$. 

The second example, due to Taniyama in \cite{T} section 18 will be discussed more briefly. Let $A / K$ be an abelian variety of dimension $g\ge 1$ and let $S$ be the set of places where $A$ has bad reduction. Let $\Ah$ be the extension of $A$ to an abelian scheme over $\spec \eo_K \setminus S$. Then $S$ is the exceptional set of the strictly compatible system $\rho = (\rho_l)$ where $\rho_l$ is the representation of $G_K$ on the Tate module $T_l A = \varprojlim_k A_{l^k} (\oK)$. Conditions (I) and (II) are known to be true with $d = 2g$ and $w = 1$ by the work of Weil \cite{We}.

As for condition (III) let $K_{\ep}$ be the completion of $K$ at the place $\ep$. Let $\oK_{\ep}$ be the algebraic closure of $K_{\ep}$ in the completion of $\oK$ with respect to a given place $\eP$ of $\oK$ above $\ep$. We write $\eo_{\ep}$ and $\oeo_{\ep}$ for the respective rings of integers. The residue field $k_{\ep} = \eo_K / \ep = \eo_{\ep} / \ep \eo_{\ep}$ has $N \ep$ elements and $\ok_{\ep} = \oeo_{\ep} / \eP \oeo_{\ep}$ is an algebraic closure of $k_{\ep}$. Now assume that $\ep \notin S$ and $\ep \mid l$. The connected-\'etale exact sequence of the $l$-divisible group $\Ah (l)$ of $\Ah$ over $\eo_{\ep}$
\[
 0 \longrightarrow \Ah (l)^0 \longrightarrow \Ah (l) \longrightarrow \Ah (l)^{\et} \longrightarrow 0
\]
 gives an exact sequence of Tate-modules with $G_{\eP} = \Gal (\oK_{\ep} / K_{\ep})$-action
\[
 0 \longrightarrow T (\Ah (l)^0) \longrightarrow T_l A \longrightarrow T_l (\Ah \otimes k_{\ep}) \longrightarrow 0 \; .
\]
Setting $E^0_l = T (\Ah (l)^0)$ and
\[
 E^{\et}_l = T_l (\Ah \otimes k_{\ep}) = \varprojlim_{\nu} \Ah (\ok_{\ep})_{l^{\nu}}
\]
the conditions in (III) are verified. 

For a scheme $\eX$ of finite type over $\spec \Z$ there is the Hasse--Weil zeta function of $\eX$
\[
 \zeta_{\eX} (s) = \prod_{x \in |\eX|} (1 - Nx^{-s})^{-1} \quad \text{for} \; \RRe s > \dim \eX \; .
\]
Here $|\eX|$ is the set of closed points of $\eX$ and for $x \in \eX$ the number of elements in the finite residue field $\kappa (x)$ is denoted by $Nx$. Using the basic relation
\begin{equation}
 \label{eq:12}
\zeta_{\Ah} (s) = \prod^{2g}_{i=0} L (\Lambda^i \rho , s)^{(-1)^i} \; ,
\end{equation}
Taniyama's theorem gives the formula
\begin{equation} \label{eq:13}
 \zeta_{\Ah} (s) = \prod^{\infty}_{n=1} L (\psi_n , s) \; .
\end{equation}
We do not work out $L (\psi_n, s)$ completely but only $L (\varphi_n , s)$ which has the same Euler factors for $\ep \nmid n$. We have $E (n) = A_n (\oK)$ and $E (n)^*$ is the subset of $A_n (\oK)$ of elements of order $n$. The $G_K$-orbits on $A_n (\oK)$ are in bijection with the closed points $P$ of $A_n$ and $K_P$ is the residue field of $P$. Set
\[
 |A^*_n| = |A_n| \setminus \bigcup_{d \mid n \atop d \neq n} |A_d| \; .
\]
Then according to \eqref{eq:5} we have the formula
\[
 L (\varphi_n ,s) = \prod_{P \in |A^*_n|} \zeta_{K_P,S} (s) \; .
\]
For $P \in |A^*_n|$ let $\tzeta_{K_P,S} (s)$ be modifications of $\zeta_{K_P,S} (s)$ at the Euler factors of primes $\ep \mid n$ such that we have
\[
 L (\psi_n,s) = \prod_{P \in |A^*_n|} \tzeta_{K_P,S} (s) \; .
\]
Then Taniyama's formula becomes
\begin{equation} \label{eq:14}
 \zeta_{\Ah} (s) = \prod_P \tzeta_{K_P,S} (s) \; .
\end{equation}
Here $P$ runs over the closed torsion points of $A$.

\section{Horizontal factorizations for group schemes}
\label{sec:3}
In this section we give a simple geometric proof of a generalization to certain commutative group schemes $\Gh$ of equations \eqref{eq:11} and \eqref{eq:14} which concerned $\G_m$ resp. $\Ah$. The set of closed points of $\Gh$ will be essentially partitioned into the set of closed points on the horizontal prime divisors in $\Gh$ obtained by taking the closures in $\Gh$ of torsion points $P$ on the generic fibre of $\Gh$. 

For a finite set $S$ of maximal ideals in $\eo_K$ set $U = \spec \eo_K \setminus S$. Thus $U = \spec \eo_{K,S}$ where
\[
 \eo_{K,S} = \Big\{ \frac{f}{g} \mid f,g \in \eo_K , \ep \nmid g \; \text{for} \; \ep \notin S \Big\} \; .
\]
We consider commutative smooth separated group schemes $\Gh$ of finite type over $U$ for which the $n$-multiplication is finite and flat for all $n \ge 1$. Thus $\G_m$ and $\Ah$ above qualify but $\G_a$ does not. It follows that $\Gh_n$ is a finite flat group scheme over $U$. For $\ep \in U$ the reduction $\Gh_{n,k_{\ep}} = \Gh_n \otimes k_{\ep}$ is an \'etale group scheme over $k_{\ep}$ if $\ep \nmid n$. For $\ep \mid n$ we have the connected-\'etale sequence
\[
 0 \longrightarrow \Gh^0_{n,k_{\ep}} \longrightarrow \Gh_{n,k_{\ep}} \longrightarrow \Gh^{\et}_{n,k_{\ep}} \longrightarrow 0 \; .
\]
It splits since $\Gh^{\red}_{n,k_{\ep}}$ is a closed subgroup scheme of $\Gh_{n,k_{\ep}}$ because $k_{\ep}$ is perfect and the composition
\[
 \Gh^{\red}_{n,k_{\ep}} \hookrightarrow \Gh_{n,k_{\ep}} \longrightarrow \Gh^{\et}_{n,k_{\ep}}
\]
is an isomorphism, c.f.~\cite{Wa} 6.8. The unit section of $\Gh$ is a regular immersion by \cite{L} 6 Proposition 3.13. Hence $\Gh_n \hookrightarrow \Gh$ is a regular immersion as well, the $n$-multiplication being flat. It follows that $\Gh_n$ is purely one-dimensional. See \cite{L} 6 Proposition 3.11 for both assertions. The irreducible components of $\Gh_n$ are therefore the closures in $\Gh$ of the closed points $P$ of $\Gh_{n,K} = \Gh_n \otimes K$. Let $C_P = \{ \oP \}$ be the irreducible component of $\Gh_n$ corresponding to $P$ with its induced reduced structure. Over $U_n = U \setminus S_n$  with $S_n = \{ \ep \mid n \}$ the scheme $\Gh_n$ is finite and \'etale. Hence the regular purely one-dimensional scheme $\Gh_n \times_U U_n$ is the disjoint union of the restrictions of the $C_P$ to $U_n$
\begin{equation} \label{eq:15}
 \Gh_n \times_U U_n = \coprod_P C_P \times_U U_n \; .
\end{equation}
Let $K_P$ be the residue field of $P$, an algebraic number field. Since $C_P \to U$ is finite and $U = \spec \eo_{K,S}$ is affine, $C_P = \spec A_P$ is affine as well. The ring $A_P$ is an order in $\eo_{K_P,S}$ and since $C_P \times_U U_n$ and hence $A_{P,S_n}$ is regular we have $\eo_{K_P , S \cup S_n} = A_{P,S_n}$. We therefore find
\begin{equation} \label{eq:16}
 C_P \times_U U_n = \spec \eo_{K_P,S \cup S_n} \; .
\end{equation}
Noting that $C_P \times_U U_n$ is even \'etale over $U_n$ it follows that the extension of number fields $K_P / K$ is unramified at all maximal ideals $\ep \notin S \cup S_n$.

The closed points of $\Gh$ are contained in the closed fibres $\Gh_{k_{\ep}} = \Gh \otimes k_{\ep}$ for maximal ideals $\ep \notin S$. A closed point $x$ of $\Gh_{k_{\ep}}$ corresponds to a finite  $\Gal (\ok_{\ep} / k_{\ep})$-orbit in $\Gh_{k_{\ep}} (\ok_{\ep}) = \Gh (\ok_{\ep})$. Since $k_{\ep}$ is finite, the group $\Gh (\ok_{\ep})$ is a union of finite abelian groups and hence every element has finite order. It follows that we have an equality
\[
 |\Gh| = \bigcup^{\infty}_{n=1} |\Gh_n| \; .
\]
For any group scheme $\Hh$ over a scheme we denote by $\Hh^*_n$ the open subscheme of $\Hh_n$ of points of order $n$:
\[
 \Hh^*_n = \Hh_n \setminus \bigcup_{d \mid n \atop d \neq n} \Hh_d \; .
\]
With this notation we get the disjoint decomposition:
\[
 |\Gh| = \coprod^{\infty}_{n=1} |\Gh^*_n| \; .
\]
This implies the formula
\begin{equation} \label{eq:17}
 \zeta_{\Gh} (s) = \prod^{\infty}_{n=1} \zeta_{\Gh^*_n} (s) \; .
\end{equation}
By equations \eqref{eq:15} and \eqref{eq:16} we have
\begin{equation}\label{eq:18}
 \Gh^*_n \times_U U_n = \coprod_Q \spec \eo_{K_Q,S\cup S_n}
\end{equation}
where $Q$ runs over the closed points of order $n$ in $\Gh_K$.

Hence
\begin{equation} \label{eq:19}
 \zeta_{\Gh^*_n \times_U U_n} (s) = \prod_Q \zeta_{K_Q,S \cup S_n} (s) \; .
\end{equation}
Changing the finitely many Euler factors at primes dividing $n$ suitably one gets modifications $\tzeta_{K_Q ,S}$ of $\zeta_{K_Q,S}$ for which the formula
\[
 \zeta_{\Gh^*_n} (s) = \prod_Q \tzeta_{K_Q,S} (s)
\]
holds. Then \eqref{eq:17} becomes an equation
\begin{equation} \label{eq:20}
 \zeta_{\Gh} (s) = \prod_P \tzeta_{K_P,S} (s)
\end{equation}
where $P$ runs over the closed torsion points of $\Gh_K$. This gives a simple geometric explanation for equation \eqref{eq:14}. To make these modifications more explicit we now discuss the zeta-function of $\Gh^*_n$ in more detail. In the decomposition
\begin{equation} \label{eq:21}
 \zeta_{\Gh^*_n} (s) = \prod_{\ep \notin S} \zeta_{\Gh^*_n \otimes k_{\ep}} (s) 
\end{equation}
the factors at primes not dividing $n$ are understood by formula \eqref{eq:19}. Let now $\ep \notin S$ be a prime with $\ep \mid n$. 
We have $\Gh^*_n \otimes k_{\ep} = (\Gh_{n,k_{\ep}})^*$. Since $k_{\ep}$ is perfect, $\Gh^{\red}_{n,k_{\ep}}$ is a closed subgroup scheme of $\Gh_{n,k_{\ep}}$ and we have
\[
 (\Gh^*_n \otimes k_{\ep} )^{\red} = \Gh^{\red *}_{n,k_{\ep}} \quad \text{and} \quad \Gh^{\red}_{n,k_{\ep}} \cong \Gh^{\et}_{n,k_{\ep}} \; .
\]
For the $\ep$-factor in equation \eqref{eq:21} we therefore find
\begin{equation} \label{eq:22}
 \zeta_{\Gh^*_n \otimes k_{\ep}} (s) = \zeta_{\Gh^{\red *}_{n,k_{\ep}} } (s) = \zeta_{\Gh^{\et *}_{n,k_{\ep}}} (s) \; . 
\end{equation}
Consider the connected-\'etale sequence for the group scheme $\Gh_{n,\eo_{\ep}} = \Gh_n \otimes_U \eo_{\ep}$
\[
 0 \longrightarrow \Gh^0_{n,\eo_{\ep}} \longrightarrow \Gh_{n,\eo_{\ep}} \longrightarrow \Gh^{\et}_{n,\eo_{\ep}} \longrightarrow 0 \; .
\]
The closed points of $\Gh^{\et}_{n,k_{\ep}}$ are in bijection with the closed points of $\Gh^{\et}_{n,K_{\ep}} = \Gh^{\et}_{n,\eo_{\ep}} \otimes_{\eo_{\ep}} K_{\ep}$ because the Galois orbits on $\Gh^{\et}_{n,\eo_{\ep}} (\oK_{\ep}) = \Gh^{\et}_{n,\eo_{\ep}} (\ok_{\ep}) = \Gh^{\et}_{n,k_{\ep}} (\ok_{\ep})$ are the same. Hence the closed points of $\Gh^{\et *}_{n,k_{\ep}}$ are in bijection with the closed points of $\Gh^{\et}_{n,K_{\ep}}$ of order $n$. Therefore the finite product of Euler factors defining $\zeta_{\Gh^*_n \otimes k_{\ep}} (s)$ can also be described in terms of the generic fibre $\Gh^{\et}_{n,K_{\ep}}$. For $\Gh = \Ah$ an abelian scheme over $U$ this is the geometry behind Taniyama's definition of the Euler factors of $L (\psi_n,s)$ over primes $\ep$ dividing $n$. 

\begin{examp}
 Consider $\Gh = \G_{m,\eo_K}$ over $U = \spec \eo_K$ and thus $\Gh_n = \mu_{n,\eo_K}$. The fibres of $\Gh^*_n$ over primes $\ep \mid n$ are empty because in this case the order of $\mu_n (\ok_{\ep})$ is less than $n$ and so there are no points of order $n$. It follows that $\Gh^*_n = \Gh^*_n \times U_n$. By formula \eqref{eq:19} we have
\begin{equation} \label{eq:23}
 \zeta_{\Gh^*_n \times U_n} (s) = \prod_Q \zeta_{K_Q , S_n} (s)
\end{equation}
where $Q$ runs over the closed points of order $n$ in $\G_{m,K}$. The zeta function $\zeta_{K_Q,S_n} (s)$ is the same as the modified Dedekind zeta function $\tzeta_{K_Q} (s)$ obtained by removing all Euler factors in $\zeta_{K_Q} (s)$ over primes dividing $n$. It follows that we have
\begin{equation} \label{eq:24}
 \zeta_{\Gh^*_n} (s) = \prod_Q \tzeta_{K_Q} (s)
\end{equation}
and therefore by \eqref{eq:17} with $P$ running over the closed torsion points of $\G_{m,K}$
\begin{equation} \label{eq:25}
 \zeta_{\G_{m,\eo_K}} (s) = \prod_P \tzeta_{K_P} (s) \; .
\end{equation}
Thus we obtain equation \eqref{eq:11} again. One can make $\Gh^*_n$ explicit as follows. Consider first the case $K = \Q$. The primitive roots of unity of order $n$ in $\oQ^{\times}$ are all conjugated by $G_{\Q}$. Hence $\Gh_{\Q} = \G_{m,\Q}$ has exactly one closed point $Q$ of order $n$. It corresponds to the maximal ideal generated by the cyclotomic polynomial $\Phi_n (T)$ in $\Q [T,T^{-1}]$. Its residue field is the cyclotomic field $\Q (\zeta_n)$. Formula \eqref{eq:18} therefore gives:
\begin{equation} \label{eq:26}
 \mu^*_{n,\Z} = \mu^*_{n,\Z} \times U_n = \spec \Z [1/n , \zeta_n ] = \spec \Z [ 1/n ] [T] / (\Phi_n (T)) \; .
\end{equation}
 The latter description for $\mu^*_{n,\Z}$ is the one used in \cite{JR} Example 3.1. Equation \eqref{eq:26} also follows from the fact that $\Phi_d$ and $\Phi_e$ are coprime in $\Z [1/n] [T]$ if $d \neq e$ and $d \mid n$ and $e \mid n$, use e.g. \cite{A}. In any case, one has
\[
 \zeta_{\mu^*_{n,\Z}} (s) = \tzeta_{\Q (\zeta_n)} (s)
\]
and therefore
\[
 \zeta_{\G_{m,\Z}} (s) = \prod^{\infty}_{n=1} \tzeta_{\Q (\zeta_n)} (s) \; .
\]
In the general case base change gives
\[
 \mu^*_{n,\eo_K} = \spec \Z [ 1/n , \zeta_n ] \otimes_{\Z} \eo_K = \spec \eo_K [ 1/n ] [T] / (\Phi_n (T)) \; .
\]
Thus the fibre over $K$ decomposes into finitelly many points $Q$ corresponding to the irreducible divisors of $\Phi_n$ over $K$. We can form their modified Dedekind zeta functions $\tzeta_{K_Q} (s)$ and from them obtain $\zeta_{ \G_{m,\eo_K}} (s)$ by formula \eqref{eq:25}.
\end{examp}

\section{Generalizations}
For a scheme $\eX$ which is separated of finite type over $U = \spec \eo_K \setminus S$ and a group scheme $\Gh / U$ as before we have
\[
 |\eX \times_U \Gh| = \coprod^{\infty}_{n=1} |\eX \times_U \Gh^*_n | \; .
\]
Hence
\[
 \zeta_{\eX \times_U \Gh} (s) = \prod^{\infty}_{n=1} \zeta_{\eX \times_U \Gh^*_n} (s) \; .
\]
Using the decomposition \eqref{eq:18} we find
\[
 \eX \times_U \Gh^*_n \times_U U_n = \coprod_Q \eX \otimes_{\eo_{K,S}} \eo_{K_Q,S \cup S_n} \; .
\]
Hence the zeta-function of $\eX \times_U \Gh^*_n$ is up to the Euler factors for $\ep \mid n$ equal to the product of the zeta functions of $\eX \otimes_{\eo_{K,S}} \eo_{K_Q,S \cup S_n}$.

For $\Gh = \G_m$ over $U = \spec \Z$ we have
\[
 \eX \times \mu^*_{n,\Z} = \eX \otimes \Z [1/n , \zeta_n]
\]
and hence
\[
 \zeta_{\eX \times \G_m} (s) = \prod^{\infty}_{n=1} \zeta_{\eX \otimes \Z [1/n , \zeta_n]} (s) \; .
\]
In \cite{JR} section 3 the zeta function of $\eX \otimes \Z [1/n, \zeta_n]$ is further decomposed into a product of zeta-functions of $\eX$ twisted by Dirichlet characters.

\subsection*{More general schemes}
Consider an open subscheme $\Uh \subset \Pa^N = \Pa^N_{\Z}$. We will describe a way to get a horizontal factorization of the zeta function of $\Uh$. Setting $\eX = \Pa^N \setminus \Uh$, the formula 
\[
\zeta_{\eX} (s) \zeta_{\Uh} (s) = \zeta_{\Pa^N} (s) =  \zeta (s) \zeta (s-1) \cdots \zeta (s-N)
\]
then gives information on the zeta function of the projective scheme $\eX$. Using the standard decomposition
\[
 \Pa^N = \A^N \amalg \ldots \amalg \A^0
\]
we have
\begin{equation} \label{eq:27}
 \zeta_{\Uh} (s) = \prod^N_{M=0} \zeta_{\Uh \cap \A^M} (s) \; . 
\end{equation}
For $I = \{ 0 , 1 \}^M$ and $i \in I$ set $A_i = A_{i_1} \times \ldots \times A_{i_M}$ where $A_0 = 0_{\A^M} = \spec \Z$ and $A_1 = \G_m$. Then we have
\[
 \A^M = \coprod_{i\in I} A_i
\]
and hence
\begin{equation} \label{eq:28}
 \zeta_{\Uh \cap \A^M} (s) = \prod_{i \in I} \zeta_{\Uh \cap A_i} (s) \; .
\end{equation}
Using the canonical isomorphism $A_i = \G^{|i|}_m$ where $|i| = i_1 + \ldots + i_M$ we will identify $\Uh \cap A_i$ with an open subscheme of $\G^{|i|}_m$. Thus we may assume that $\Uh$ itself is an open subscheme of $\Gh = \G^N_m$ for some $N$. From the decomposition
\[
 |\Gh| = \coprod^{\infty}_{n=1} |\Gh^*_n|
\]
we obtain
\[
  |\Uh| = \coprod^{\infty}_{n=1} |\Uh \cap \Gh^*_n|
\]
and hence
\[
 \zeta_{\Uh} (s) = \prod^{\infty}_{n=1} \zeta_{\Uh \cap \Gh^*_n} (s) \; .
\]
Next, note that
\[
 \Gh^*_n = \coprod_{j\in J_n} \mu^*_{j_1} \times \ldots \times \mu^*_{j_N} \; ,
\]
where $J_n$ is the set of $N$-tuples $j = (j_1 , \ldots , j_N)$ of divisors of $n$ whose smallest common multiple is $n$. Thus
\begin{equation} \label{eq:29}
 \zeta_{\Uh \cap \Gh^*_n} (s) = \prod_{j\in J_n} \zeta_{\Uh \cap (\mu^*_{j_1} \times \ldots \times \mu^*_{j_N})} (s) \; .
\end{equation}
We have seen that $\mu^*_{\nu} = \spec \Z [1/\nu , \zeta_{\nu}]$ where $\zeta_{\nu}$ is a primitive $\nu$-th root of unity. This gives
\[
 \Uh \cap (\mu^*_{j_1} \times \ldots \times \mu^*_{j_N}) = \Uh \otimes_{\Z} R_j
\]
where
\[
 R_j = \Z [\zeta_{j_1}] \otimes_{\Z} \ldots \otimes_{\Z} \Z [\zeta_{j_N}] \otimes_{\Z} \Z [1/n]
\]
is an \'etale $\Z [1/n]$-algebra. We have
\[
 \spec R_j = \coprod_{\alpha} \spec \eo_{K_{\alpha}} [1/n]
\]
where $\alpha$ runs over the maximal ideals of $\Q (\zeta_{j_1}) \otimes \ldots \otimes \Q (\zeta_{j_N})$. They correspond to the orbits of the diagonal $G_{\Q}$-action on $\mu^*_{j_1} (\oQ) \times \ldots \times \mu^*_{j_N} (\oQ)$. It follows that
\begin{equation} \label{eq:30}
 \zeta_{\Uh \cap (\mu^*_{j_1} \times \ldots \times \mu^*_{j_N})} (s) = \prod_{\alpha} \zeta_{\Uh_{\alpha}} (s)
\end{equation}
where $\Uh_{\alpha} = \Uh \cap \spec \eo_{K_{\alpha}} [1/n]$ is an open subscheme of $\spec \eo_{K_{\alpha}} [1/n]$. Hence, up to finitely many Euler factors, $\zeta_{\Uh_{\alpha}} (s)$ is the Dedekind zeta-function $\zeta_{K_{\alpha} , S_n} (s)$. Thus, combining formulas \eqref{eq:27}--\eqref{eq:30}, we obtain the desired horizontal factorization of $\zeta_{\Uh} (s)$.

\end{document}